\documentclass[10pt,reqno,a4paper]{amsart}
\usepackage{amsmath}
\usepackage{enumerate}
\makeatletter
\@namedef{subjclassname@2010}{%
  \textup{2010} Mathematics Subject Classification}
\makeatother
\newtheorem{theorem}{Theorem}[section]
\newtheorem{corollary}{Corollary}[section]
\newtheorem{lemma}{Lemma}[section]

\theoremstyle{definition}

\theoremstyle{remark}

\numberwithin{equation}{section}

\title[]
      {The explicit evaluations formula for Ramanujan's singular moduli and Ramanujan- Selberg continued fraction}

\author[D. J. Prabhakaran]{D. J. Prabhakaran}
\address{Department of Mathematics \\
Anna University, MIT Campus\\
Chennai- 600025\\ India}
\email{asirprabha@gmail.com}
\author[K. Ranjith kumar]{K. Ranjith kumar}
\address{Department of Mathematics \\
Anna University, MIT Campus\\
Chennai- 600025\\ India}
\email{ranjithkrkkumar@gmail.com}
\begin{document}
\begin{abstract}
At scattered places of his notebooks, Ramanujan recorded over 30 values of singular moduli $\alpha_n$. All those results were proved
by Berndt et. al by employing Weber-Ramanujan's class invariants. In this paper, we initiate to derive the explicit evaluations formula for $\alpha_{9n}$ and $\alpha_{n/9}$ by involving class invariant. For this purpose, we establish several new $P-Q$ mixed modular equations involving theta-functions. Further application of these modular equations, we derive a new formula to explicit evaluation of Ramanujan- Selberg continued fraction.
\end{abstract}


\maketitle
\section{Introduction}
Ramanujan's general theta-function \cite{Berndt-notebook-3} is defined by
\begin{eqnarray*}
f(a,b)&=&\sum^{n=\infty}_{n=-\infty}a^{n(n+1)/2}b^{n(n-1)/2}, \ |ab|<1.
\end{eqnarray*}
By using Jacobi's fundamental factorization formula, the above theta function takes the form as follows :
\begin{eqnarray*}
f(a,b)&=&(-a;ab)_\infty(-b;ab)_\infty(ab;ab)_\infty.
\end{eqnarray*}
The following definitions of theta functions $\varphi$, $\psi$ and $f$ with $|q|<1$ are classical:
\begin{eqnarray}
\varphi(q) &=& (q,q)=\sum^{\infty}_{n=-\infty}q^{n^2}=(-q;q^2)^2_{\infty}(q^2;q^2)_{\infty},\\
\psi(q) &=& (q,q^3)=\sum^{\infty}_{n=0}q^{n(n+1)/2}=\frac{(q^2;q^2)_{\infty}}{(q;q^2)^2_{\infty}},\\
f(-q)&=&f(-q,-q^2)=\sum^{\infty}_{n=-\infty}(-1)^nq^{n(3n-1)/2}=(q;q)_{\infty},\label{classicaltheta}
\end{eqnarray}
where, $(a;q)_{\infty}=\prod^{\infty}_{n=0}\left(1-aq^n\right)$.\\
If $q=e^{2\pi i\tau},$ then, $f(-q)=q^{-1/24}\eta(\tau),$ where, $\eta(\tau)$ is classical Dedekind eta-function.\\
The ordinary or Gaussian hypergeometric function is defined by
\begin{eqnarray*}
{_2F_1}\left(
\begin{array}{ccc}
  a, & b \\
  c &
\end{array} ; \displaystyle z
\right)
&=&\sum_{n=0}^{\infty}\frac{(a)_n(b)_n}{(c)_n n!}z^n, \ \ \ \ |z|< 1
\end{eqnarray*}
where,
\begin{eqnarray*}
(a)_n&=&\left\{ \begin{array}{lll}
 a(a+1)(a+2)\cdots(a+n-1), \  n=1,2,3,...\\
 1,   \    n=0 . &
\end{array} \right.
\end{eqnarray*}
Now, we shall recall the definition of modular equation from \cite{Berndt-notebook-3}. The complete elliptic integral of the first kind $K(k)$ of modulus $k$ is defined by
\begin{align}\label{112235eq0}
K(k)=\int_{0}^{\frac{\pi}{2}}\frac{d\theta}{\sqrt{1-k^2\sin^2 \theta}}=\frac{\pi}{2}\sum^{\infty}_{n=0}\frac{\left(\frac{1}{2}\right)^2_n}{(n!)^2}k^{2n}=
\frac{\pi}{2}\varphi^2\left(e^{-\pi\frac{K'}{K}}\right),
\quad (0<k<1)
\end{align}
and let $K'=K(k'),$ where $k'=\sqrt{1-k^2}$ is represented as the complementary modulus of $k$. Let $K, K', L,$ and $L'$ denote the complete elliptic integrals of the first kind associated with the moduli $k, k', l,$ and $l'$ respectively. In case, the equality
\begin{equation} \label{eq00}
n\frac{K'}{K}=\frac{L'}{L}
\end{equation}
holds for a positive integer $n$, then a modular equation of degree $n$ is the relation between the moduli $k,$ and $l$, which is implied by equation (\ref{eq00}). Ramanujan defined his modular equation involving $\alpha,$ and $\beta$, where, $\alpha=k^2$ and $\beta=l^2$. Then we say $\beta$ is of degree $n$ over $\alpha$ if
\begin{eqnarray*}
\frac{{_2F_1}\left(1/2, 1/2; 1; 1-\beta\right)}{{_2F_1}\left(1/2, 1/2; 1; \beta\right)} &=& n\frac{{_2F_1}\left(1/2, 1/2; 1; 1-\alpha\right)}{{_2F_1}\left(1/2, 1/2; 1; \alpha\right)}
\end{eqnarray*}
Ramanujan recorded various degrees of modular equations in his note books. For instance, the modular equation of degree 3 (see \cite[Entry 5(ii), (ix), Ch.19, p.230]{Berndt-notebook-3}) as following
\begin{eqnarray}\label{lemma 2}
(\alpha\beta)^{1/4}+((1-\alpha)(1-\beta))^{1/4}&=& 1, \\
(\alpha(1-\beta))^{1/2} + (\beta(1-\alpha))^{1/2} &=& 2(\alpha\beta(1-\alpha)(1-\beta))^{1/8}. \label{lemma 2.e}
\end{eqnarray}
Also, Ramanujan defined mixed modular equation or modular equation of composite degrees, along with four distinct moduli.
Ramanujan recorded 23 $P$-$Q$ modular equations in terms of their theta function in his notebooks \cite{sr}. All those proved by Berndt et al. by employing the theory of theta functions and modular forms.\\
If, as usually quoted in the theory of elliptic functions, $k = k(q)$ denotes the modulus, then, the singular moduli $k_n$ is defined by $k_n = (e^{-\pi\sqrt{n}})$, where $n$ is a positive integer. In terms of Ramanujan, set $\alpha = k^2$ and $\alpha_n = k^2_n,$ he hypothesized the values of over 30 singular moduli in his notebooks. Also, he asserted the value of $k^2_{210}$ (which he wrote in his second letter to Hardy \cite[p. xxix]{13}), which was earlier proved by Wastson \cite{17} by using a remarkable formula. The formula found can also be in Ramanujan's first notebook \cite[Vol. 1, p. 320]{sr}. The same formula can also be used to evaluate various values of $\alpha_n$ for even values of $n$. On page 82 of his first notebook, Ramanujan stated three additional theorems for calculating $\alpha_n$ for even values of $n$. Particularly, he offered formulae for $\alpha_{4p}$, $\alpha_{8p},$ and $\alpha_{16p}$. Moreover, he recorded several values of $\alpha_n$ for odd values of $n$ in his first and second notebook. All these results were proved by Berndt et al. by employing  Ramanujan's class invariants $G_n$ and $g_n$. The Ramanujan's class invariants \cite[p.183, (1.3)]{Berndt-notebook-5} defined by for $q=e^{-\pi\sqrt{n}}$,
\begin{align}\label{evq1}
G_n = 2^{-1/4}q^{-1/24} \chi(q) = \frac{f(q)}{2^{1/4}q^{1/24}f(-q^2)}
\quad ; \quad
g_n = 2^{-1/4}q^{-1/24} \chi(-q) = \frac{f(-q)}{2^{1/4}q^{1/24}f(-q^2)},
\end{align}
where, $n$ is a positive rational number and $\chi(q)=(-q;q^2)_\infty$. Ramanujan evaluated a total of 116 class invariants \cite[p.189-204]{Berndt-notebook-5}. These class invariants were proved by various authors using techniques such as modular equations, Kronecker limit formula, and empirical process (established by Watson) \cite[Chapter 34]{Berndt-notebook-5}.

By Entry 12(v), (vi)\cite[p. 124]{Berndt-notebook-3}, and our conclusion \eqref{evq1} is that
\begin{align}\label{eq1}
G_n = (4\alpha_n(1-\alpha_n))^{-1/24}
\quad ; \quad
g_n = (4\alpha_n(1-\alpha_n)^{-2})^{-1/24} .
\end{align}
It follows the concept that if we know the explicit values of $G_n$ and $g_n,$ then the corresponding values of $\alpha_n$ can be obtained by solving a quadratic equations. However, the expression one obtain are generally unattractive, and so, better algorithm can be sought for evaluations of $\alpha_n$.

Let, for $|q| < 1$,
\begin{eqnarray}\label{ciontu}
N(q) : &=& 1+\frac{q}{1}_+\frac{q+q^2}{1}_+\frac{q^3}{1}_+\frac{q^2+q^4}{1}_+\cdots\cdots
\end{eqnarray}
In his notebooks \cite[p. 290]{sr}, Ramanujan recorded that
\begin{eqnarray}\label{ciontu2}
N(q)&=& \frac{(-q ; q^2)_{\infty}}{(-q^2 ; q^2)_{\infty}}.
\end{eqnarray}
This formula was at first proved by Selberg \cite{sel} and the alternative proof was given by Ramanathan \cite{kgr4}.  From \eqref{ciontu} and \eqref{ciontu2}, we obtain that
\begin{eqnarray} \nonumber
 S_1(q): &=& \frac{q^{1/8}}{N(q)}
 =\frac{q^{1/8}}{1}_+\frac{q}{1}_+\frac{q+q^2}{1}_+\frac{q^3}{1}_+\frac{q^2+q^4}{1}_+\cdots\cdots, \\
 &=&\frac{q^{1/8}(-q^2 ; q^2)_{\infty}}{(-q ; q^2)_{\infty}}\label{eq0.1},
\end{eqnarray}
which is called the Ramanujan-Selberg continued fraction.  Closely related to $ S_1(q)$ is the continued fraction $ S_2(q)$ (see \cite[(1.8), (1.9)]{zhan}), which can be defined by
\begin{eqnarray} \label{eqksn0.1}
 S_2(q): =\frac{q^{1/8}}{1}_+\frac{-q}{1}_+\frac{-q+q^2}{1}_+\frac{-q^3}{1}_+\frac{q^2+q^4}{1}_+\cdots\cdots  = \frac{q^{1/8}(-q^2 ; q^2)_{\infty}}{(q ; q^2)_{\infty}}, \ \   |q|<1
\end{eqnarray}
Zhang \cite{zhan} recorded a general formula to find the explicit evaluations of \eqref{eq0.1} by using Ramanujan's  singular moduli as follows:
\begin{eqnarray} \label{2.4}
 S_1\left(e^{-\pi\sqrt{n}}\right) &=& \frac{\alpha_n^{1/8}}{\sqrt{2}}.
\end{eqnarray}
Moveover, he established a general formula to calculate $ S_2(q)$ in terms of singular moduli.

The present paper is organized as follows. In Section 2, we state that a few lemmas which are essentials to prove our main results. We establish several new $P-Q$ mixed modular equations involving theta-functions, which are presented in Section 3. As application of those modular equations, we derive some new general formulae involving Weber-Ramanujan's class invariants for explicit evaluations of $\alpha_{9n}$,
$\alpha_{n/9}$, $S_1(q),$ and $ S_2(q),$ which are discussed in Section 4. Finally, calculate some explicit values of singular moduli and Ramanujan- Selberg continued fraction in Section 5.
\section{Preliminaries}
We list a few identities which are useful in establishing our main results.
\begin{lemma}\cite[Entry 12 (i), (ii), (iv) Ch.17, p.124]{Berndt-notebook-3} We have
\begin{eqnarray}
f(q)&=&\sqrt{z} 2^{-1/6}\left(\alpha(1-\alpha)/q\right)^{1/24} \label{feq1},\\
f(-q)&=&\sqrt{z} 2^{-1/6}\left(\alpha(1-\alpha)^4/q\right)^{1/24}\label{feq2},\\
f(-q^4)&=&\sqrt{z}2^{-2/3}\left(\alpha^4(1-\alpha)/q^4\right)^{1/24}\label{feq3}.
\end{eqnarray}
\end{lemma}
\begin{lemma}\cite[ p.39, Entry 24 (iii),(iv)]{Berndt-notebook-3} We have
\begin{eqnarray}\label{entry24}
\frac{f(q)}{f(-q^2)}&=& \frac{f(-q^2)f(-q^2)}{f(-q)f(-q^4)}.
\end{eqnarray}
\end{lemma}
\begin{lemma}\cite[Theorem 3.1(3.2)]{naikamahadeva} \label{lemma 6}
Let
\begin{align*}
K = \left(256\alpha\beta\gamma\delta(1-\alpha)(1-\beta)(1-\gamma)(1-\delta)\right)^{1/24}
\quad ; \quad
R = \left(\frac{\gamma\delta(1-\gamma)(1-\delta)}{\alpha\beta(1-\alpha)(1-\beta)}\right)^{1/24}.
\end{align*}
If $\alpha, \beta, \gamma, \delta$ is of degree $1,3,3,9,$  then,
\begin{eqnarray}\label{lemma 6.2}
8\left(K^3+\frac{1}{K^3}\right)\left(R^3+\frac{1}{R^3}+1\right) &=& \left(R^9+\frac{1}{R^9}\right)+10\left(R^6+\frac{1}{R^6}\right)
+19\left(R^3+\frac{1}{R^3}\right)+36.
\end{eqnarray}
\end{lemma}

\section{New Mixed Modular Equations}
In this section, we discuss about the establishment of a few novel mixed modular equations using theory of theta functions.
\begin{theorem}\label{the1}
If  $P=\displaystyle \frac{f(-q)f(-q^{3})}{q^{1/2}f(-q^{4})f(-q^{12})}$  and  $ Q=\displaystyle \frac{f(-q^{3})f(-q^{9})}{q^{3/2}f(-q^{12})f(-q^{36})}$, then,
\begin{eqnarray*}
\left(PQ+\frac{16}{PQ}\right)\left(\frac{P}{Q}+\frac{Q}{P}+1\right) &=&\left(\frac{P}{Q}+\frac{Q}{P}\right)^3
-2\left(\frac{P}{Q}+\frac{Q}{P}\right)^2-8\left(\frac{P}{Q}+\frac{Q}{P}\right)-8.
\end{eqnarray*}
\end{theorem}
\begin{proof}
Transcribing $P$ and $Q$ using \eqref{feq2} and \eqref{feq3}, then simplifying, we arrive at
\begin{align}\label{eqn1}
\frac{P}{2}=\left(\frac{(1-\alpha)(1-\beta)}{\alpha\beta}\right)^{1/8}
\quad ; \quad
\frac{Q}{2}=\left(\frac{(1-\gamma)(1-\delta)}{\gamma\delta}\right)^{1/8}.
\end{align}
Employing \eqref{eqn1} in $\eqref{lemma 2}$, we have
\begin{align}\label{e4qn1}
K=\left(\frac{8PQ}{(P^2+4)(Q^2+4)}\right)^{1/3}
\quad ; \quad
R=\left(\frac{P^2Q+4Q}{PQ^2+4P}\right)^{1/3},
\end{align}
where, $K,$ and $R$ are defined as in Lemma \ref{lemma 6}. Now, on applying  \eqref{e4qn1} in \eqref{lemma 6.2}, we obtain that
\begin{eqnarray}\label{5eqnnn}
A(P,Q)B(P,Q) &=& 0,
\end{eqnarray}
where,
\begin{eqnarray*}
A(P,Q) &=&P^6Q^6-8P^5Q^5-64P^3Q^5-80P^4Q^4-256P^2Q^4-64P^5Q^3-768P^3Q^3
\\ && -1024PQ^3-256P^4Q^2-1280P^2Q^2-1024P^3Q-2048PQ+4096,\\
B(P,Q) &=& {Q}^{6}-{P}^{3}\,{Q}^{5}-2\,P\,{Q}^{5}-{P}^{4}\,{Q}^{4}-5\,{P}^{2}\,{Q}^{4}-{P}^{5}\,{Q}^{3}-12\,{P}^{3}\,{Q}^{3}-16\,P\,{Q}^{3}
\\ && -5\,{P}^{4}\,{Q}^{2}-16\,{P}^{2}\,{Q}^{2}-2\,{P}^{5}\,Q-16\,{P}^{3}\,Q+{P}^{6}.
\end{eqnarray*}
We observe that the first factor of \eqref{5eqnnn} does not vanish for $q\rightarrow 0$. Nevertheless, the second factor vanish for that specific value. Dividing them by $P^3Q^3$ and rearranging the terms. Hence we complete the proof.
\end{proof}
\begin{theorem}\label{th5ue1}
If  $P=\displaystyle \frac{q^{1/4}f(-q)f(-q^{12})}{f(-q^{3})f(-q^{4})}$  and  $ Q=\displaystyle \frac{q^{3/4}f(-q^{3})f(-q^{36})}{f(-q^{9})f(-q^{12})}$, then
\begin{eqnarray}\label{ehma}
\left(\frac{1}{PQ}+PQ\right)^2 &=&\left(\frac{P}{Q}+\frac{Q}{P}\right)^2\left(\frac{1}{PQ}+PQ + 1\right)+4.
\end{eqnarray}
\end{theorem}
\begin{proof}
Transcribing $P$ and $Q$ using \eqref{feq2}, and \eqref{feq3}, then by simplifying, we arrive at
\begin{align}\label{alpbeta}
P=\left(\frac{\beta(1-\alpha)}{\alpha(1-\beta)}\right)^{1/8}
\quad ; \quad
Q=\left(\frac{\delta(1-\gamma)}{\gamma(1-\delta)}\right)^{1/8}.
\end{align}
Employing \eqref{alpbeta} in \eqref{lemma 2.e}, we have
\begin{align}\label{alp2beta}
\left(\alpha\beta(1-\alpha)(1-\beta)\right)^{1/8} = \frac{2P^2}{P^4+1}
\quad ; \quad
\left(\gamma\delta(1-\gamma)(1-\delta)\right)^{1/8} = \frac{2Q^2}{Q^4+1}.
\end{align}
Now applying  \eqref{alp2beta} in \eqref{lemma 6.2}, we obtain the following result
\begin{eqnarray*}\label{5e9n}
(P^2Q^6-P^5Q^5+PQ^5+2P^4Q^4+Q^4+4P^3Q^3+P^6Q^2+2P^2Q^2+P^5Q-PQ+P^4)&& \\
\times(P^2Q^6+P^5Q^5-PQ^5+2P^4Q^4+Q^4-4P^3Q^3+P^6Q^2+2P^2Q^2-P^5Q+PQ+P^4)&& \\
\times(P^4Q^6-P^5Q^5+PQ^5+P^6Q^4+2P^2Q^4-4P^3Q^3+2P^4Q^2+Q^2+P^5Q-PQ+P^2)&& \\
\times(P^4Q^6+P^5Q^5-PQ^5+P^6Q^4+2P^2Q^4+4P^3Q^3+2P^4Q^2+Q^2-P^5Q+PQ+P^2)&=& 0.
\end{eqnarray*}
We observe that the first factors of the aforemention equation vanish for $q\rightarrow 0,$ whereas, the other factors does not vanish for that specific value. Thus, we obtain that
\begin{eqnarray*}\label{5e9n}
P^2Q^6-P^5Q^5+PQ^5+2P^4Q^4+Q^4+4P^3Q^3+P^6Q^2+2P^2Q^2+P^5Q-PQ+P^4&=& 0.
\end{eqnarray*}
Dividing the equation by $P^3Q^3$ and rearranging the terms, we arrive at the desired result.
\end{proof}

\begin{theorem}\label{t3o6ohe3}
If  $P=\displaystyle \frac{f(-q)}{q^{1/8}f(-q^{4})}$ and $ Q=\displaystyle \frac{f(-q^9)}{q^{9/8}f(-q^{36})},$ then,
\begin{eqnarray}\nonumber
\left(P^4Q^4+\frac{256}{P^4Q^4}\right)\left(\frac{P}{Q}+\frac{Q}{P}+1\right) &=&\left(\frac{P}{Q}+\frac{Q}{P}\right)^6-8\left(\frac{P}{Q}+\frac{Q}{P}\right)^5+4\left(\frac{P}{Q}+\frac{Q}{P}\right)^4
+64\left(\frac{P}{Q}+\frac{Q}{P}\right)^3\\ &&-16\left(\frac{P}{Q}+\frac{Q}{P}\right)^2 -160\left(\frac{P}{Q}+\frac{Q}{P}\right)-96.\label{t3oohe3}
\end{eqnarray}
\end{theorem}
\begin{proof}
By Theorem \ref{the1} can be written in the following form,
\begin{eqnarray*}\label{12the13351}
\left(\frac{f(-q)f(-q^9)f^2(-q^3)}{qf(-q^4)f(-q^{36})f^2(-q^{12})}\right)+
16\left(\frac{qf(-q^4)f(-q^{36})f^2(-q^{12})}{f(-q)f(-q^9)f^2(-q^3)}\right) &=& \frac{h^3-2h^2-8h-8}{h+1},
\end{eqnarray*}
where, $\displaystyle h = P/Q+Q/P$. Solving the above equation, we arrive at
\begin{eqnarray}\label{12tthe13351}
\frac{f(-q)f(-q^9)f^2(-q^3)}{qf(-q^4)f(-q^{36})f^2(-q^{12})}&=& \frac{h^3-2h^2-8h-8+vh}{2(h+1)},
\end{eqnarray}
where $\displaystyle v = \pm \sqrt{h^4-4h^3-12h^2+16h+32}$. Similarly, by Theorem \ref{th5ue1}, we obtain that
\begin{eqnarray}\label{12themm51}
\left(\frac{f(-q)f(-q^9)f^2(-q^{12})}{q^{1/4}f(-q^4)f(-q^{36})f^2(-q^{3})}\right)^2 &=& \frac{h^2-2h-6+v}{2(h+1)},
\end{eqnarray}
Now multiplying \eqref{12tthe13351} and \eqref{12themm51}, then employing the value of $v^2$, we deduce that \\

$\displaystyle 4h^8-24h^7+v(4h^6-16h^5-44h^4+72h^3+256h^2+224h+64)
-44h^6+256h^5+464h^4+(-8P^4Q^4-512)h^3+(-24P^4Q^4-1728)h^2+(-24P^4Q^4-1408)h-8P^4Q^4-384 = 0$ \\

Isolating the terms containing $v$ on one side of the above equation and squaring both sides, we arrive at \\

$\displaystyle (h+1)^5(P^4Q^4h^6-8P^4Q^4h^5+4P^4Q^4h^4+64P^4Q^4h^3 -16P^4Q^4h^2-P^8Q^8h-160P^4Q^4h-256h -P^8Q^8 \\ -96P^4Q^4-256) = 0$

We observe that the second factors of above equation vanish for $q\rightarrow 0$ and the first factor does not vanish for that specific value. Dividing the aforemention equation by $P^4Q^4$ and rearranging the terms. Hence we complete the proof.
\end{proof}
\begin{theorem}\label{the3}
If  $P=\displaystyle \frac{f(-q)f(-q^{9})}{q^{5/4}f(-q^{4})f(-q^{36})}$ and $ Q=\displaystyle \frac{f(-q^{3})f(-q^{3})}{q^{3/4}f(-q^{12})f(-q^{12})},$ then
\begin{eqnarray}\label{t3he3}
\left(\frac{P}{Q}+\frac{Q}{P}\right)^3 &=& Q^2 + \frac{16}{Q^2}.
\end{eqnarray}
\end{theorem}
\begin{proof}
Theorem \ref{th5ue1} can be written in the form
\begin{eqnarray}\label{eq7.7}
u^2-(u+1)\left(P/Q+Q/P\right)^2-4&=& 0,
\end{eqnarray}
where, $\displaystyle u= \frac{qf(-q)f(-q^{36})}{f(-q^4)f(-q^{9})}+\frac{f(-q^4)f(-q^{9})}{qf(-q)f(-q^{36})}.$\\
Solving \eqref{eq7.7} for $u$ and choosing the appropriate root, then employing in \eqref{t3oohe3} we obtain, after a straightforward lengthly calculation that
\begin{eqnarray*}\nonumber
(Q^6-P^3Q^5+3P^2Q^4+3P^4Q^2-16P^3Q+P^6)(Q^6+P^3Q^5+3P^2Q^4+3P^4Q^2+16P^3Q+P^6) && \\ \times
(PQ^6-16Q^5+3P^3Q^4+3P^5Q^2-P^8Q+P^7)(PQ^6+16Q^5+3P^3Q^4+3P^5Q^2+P^8Q+P^7) &=& 0. \label{eq3.7}
\end{eqnarray*}
We observe that the first factors of above equation vanish for $q\rightarrow 0$ and other factors does not vanish for that specific value. Dividing by $P^3Q^3$ and rearranging the terms. Hence we complete the proof.
\end{proof}
\section{General formulae and the Explicit evaluations }
In this section, we establish the formulae involving Weber-Ramanujan's class invariants for explicit evaluations of $\alpha_{9n}$, $\alpha_{n/9},$ $S_1\left(q\right),$ and $S_2\left(q\right)$ by modular modular equations to those derived in pervious section.
\begin{theorem}\label{he1nbx351}
If $g_n$ is defined as in \eqref{evq1} respectively,  then
\begin{eqnarray}\nonumber
\alpha_{9n}&=& \left(\sqrt{g^{24}_n+1}-g^{12}_n\right)^2\left(\sqrt{g^{8}_n+1}-g^{4}_n\right)^4 \\&& \times
\left(\sqrt{\frac{g^8_n+1+\sqrt{g^{16}_n-g^{8}_n+1}}{2}}-\sqrt{\frac{g^8_n-1+\sqrt{g^{16}_n-g^{8}_n+1}}{2}}\right)^8 \label{thmp1eqn2 }, \\ \nonumber
\alpha_{n/9}&=&\left(\sqrt{g^{24}_n+1}-g^{12}_n\right)^2\left(\sqrt{g^{8}_n+1}-g^{4}_n\right)^4 \\&& \times
\left(\sqrt{\frac{g^8_n+1+\sqrt{g^{16}_n-g^{8}_n+1}}{2}}+\sqrt{\frac{g^8_n-1+\sqrt{g^{16}_n-g^{8}_n+1}}{2}}\right)^8 \label{them1eqn22 }.
\end{eqnarray}
\end{theorem}
\begin{proof}
Combining \eqref{feq1} and \eqref{feq3} with $q = e^{-\pi\sqrt{n}}$, then simplifying, we obtain that
\begin{align}\label{e87q1}
\alpha_n = \left(\frac{f(q)}{2^{1/2}q^{1/8}f(-q^4)}\right)^{-8}.
\end{align}
Employing \eqref{entry24} in Theorem 3.2.2 \cite[p.21]{Yi-Thesis-2000} along with replacing  $q$ by $q^3$, we obtain that
\begin{eqnarray}\label{aph1}
\left(\frac{f(-q^3)}{q^{3/8}f(-q^{12})}\right)^4 + 16\left(\frac{q^{3/8}f(-q^{12})}{f(-q^3)}\right)^4&=& \left(\frac{f(q^3)}{q^{1/8}f(-q^{6})}\right)^{12}.
\end{eqnarray}
Replacing $q$ by $-q$ in \eqref{t3he3}, then applying \eqref{aph1}, and \eqref{e87q1} with $q = e^{-\pi\sqrt{n/9}}$, we deduce that
\begin{eqnarray*}\label{them1eqn2 }
\left(\frac{\alpha^2_n}{\alpha_{9n}\alpha_{n/9}}\right)^{1/8}-
\left(\frac{\alpha_{9n}\alpha_{n/9}}{\alpha^2_n}\right)^{1/8}&=& 2g^4_n .\label{them1eqn23 }
\end{eqnarray*}
On solving the above equation  and choosing the appropriate root, then we arrive at
\begin{eqnarray}
\alpha_{9n}\alpha_{n/9}&=& \alpha^2_n\left(\sqrt{g^{8}_n+1}-g^{4}_n\right)^8. \label{t1eqn2}
\end{eqnarray}
We observed that some representation for $\displaystyle \alpha_n$ in terms of $g_n$. This is given by \cite[p.289, Eq.(9.27)]{Berndt-notebook-5}
\begin{eqnarray}\label{them1eqn3 }
\frac{1}{\sqrt{\alpha_n}}-\sqrt{\alpha_n} &=& 2g^{12}_n.
\end{eqnarray}
Employing \eqref{them1eqn3 } in \eqref{t1eqn2}, we conclude that
\begin{eqnarray}\label{t1keqn2}
\alpha_{9n}\alpha_{n/9}&=& \left(\sqrt{g^{24}_n+1}-g^{12}_n\right)^4\left(\sqrt{g^{8}_n+1}-g^{4}_n\right)^8.
\end{eqnarray}
By Theorem \ref{th5ue1},  we obtain that
\begin{eqnarray}
l^2-2^{4/3}\left(\frac{1}{\sqrt{\alpha_n}}-\sqrt{\alpha_n}\right)^{2/3}
\left(l-1\right)-4&=&0, \label{them1eqn23}
\end{eqnarray}
where, $\displaystyle l = \left(\alpha_{9n}/\alpha_{n/9}\right)^{1/8}+\left(\alpha_{n/9}/\alpha_{9n}\right)^{1/8}$. Now, employing \eqref{them1eqn3 } in \eqref{them1eqn23}, then solving for $l$ and  choosing positive real root, we deduce that
\begin{eqnarray}
\left(\frac{\alpha_{9n}}{\alpha_{n/9}}\right)^{1/8}+\left(\frac{\alpha_{n/9}}{\alpha_{9n}}\right)^{1/8}&=& 2\left(g^8_n+\sqrt{g^{16}_n-g^{8}_n+1}\right)
.\label{them1eqn23 }
\end{eqnarray}
On solving the above equation and choosing the appropriate root,  we obtain that
\begin{eqnarray}
\frac{\alpha_{9n}}{\alpha_{n/9}}&=&
\left(\sqrt{\frac{g^8_n+1+\sqrt{g^{16}_n-g^{8}_n+1}}{2}}-\sqrt{\frac{g^8_n-1+\sqrt{g^{16}_n-g^{8}_n+1}}{2}}\right)^{16}. \label{t31meqn2}
\end{eqnarray}
By combining \eqref{t1keqn2} and \eqref{t31meqn2}, this completes the proof.
\end{proof}
\begin{corollary}\label{nb1}
If $\displaystyle S_1\left(q\right)$ and $g_n$ are defined as in \eqref{eq0.1}, and \eqref{evq1} respectively,  then
\begin{eqnarray}\nonumber
S_1\left(e^{-\pi3\sqrt{n}}\right)&=& \frac{1}{\sqrt{2}}\left(\sqrt{g^{24}_n+1}-g^{12}_n\right)^{1/4}
\left(\sqrt{g^{8}_n+1}-g^{4}_n\right)^{1/2} \\&& \times
\left(\sqrt{\frac{g^8_n+1+\sqrt{g^{16}_n-g^{8}_n+1}}{2}}-\sqrt{\frac{g^8_n-1+\sqrt{g^{16}_n-g^{8}_n+1}}{2}}\right), \label{thcv}\\ \nonumber
S_1\left(e^{-\pi\sqrt{n}/3}\right)&=&\frac{1}{\sqrt{2}}\left(\sqrt{g^{24}_n+1}-g^{12}_n\right)^{1/4}
\left(\sqrt{g^{8}_n+1}-g^{4}_n\right)^{1/2} \\&& \times
\left(\sqrt{\frac{g^8_n+1+\sqrt{g^{16}_n-g^{8}_n+1}}{2}}+\sqrt{\frac{g^8_n-1+\sqrt{g^{16}_n-g^{8}_n+1}}{2}}\right). \label{thcv1}
\end{eqnarray}
\end{corollary}
\begin{proof}
Employing pervious theorem in \eqref{2.4}, it is not difficult to deduce to our corollary.
\end{proof}
\begin{theorem}\label{he1nb351}
If $\displaystyle S_2\left(q\right)$ and $G_n$ are defined as in \eqref{eqksn0.1}, and \eqref{evq1} respectively,  then
\begin{eqnarray}\nonumber
S_2\left(e^{-\pi3\sqrt{n}}\right)&=&\frac{1}{\sqrt{2}}\left(G^{12}_n-\sqrt{G^{24}_n-1}\right)^{1/4}\left(G^{4}_n-\sqrt{G^{8}_n-1}\right)^{1/2}\\&&
\times \left(\sqrt{\frac{G^8_n+1+\sqrt{G^{16}_n+G^{8}_n+1}}{2}}-\sqrt{\frac{G^8_n-1+\sqrt{G^{16}_n+G^{8}_n+1}}{2}}\right)\label{a1eqn2}, \\ \nonumber
S_2\left(e^{-\pi\sqrt{n}/3}\right)&=&\frac{1}{\sqrt{2}}\left(G^{12}_n-\sqrt{G^{24}_n-1}\right)^{1/4}\left(G^{4}_n-\sqrt{G^{8}_n-1}\right)^{1/2}\\&&
\times \left(\sqrt{\frac{G^8_n+1+\sqrt{G^{16}_n+G^{8}_n+1}}{2}}+\sqrt{\frac{G^8_n-1+\sqrt{G^{16}_n+G^{8}_n+1}}{2}}\right)\label{a1eqn21}.
\end{eqnarray}
\end{theorem}
\begin{proof}
The proof of our theorem can be obtained by Theorem \ref{th5ue1}, and Theorem \ref{the3}. Since the proof is analogous to Theorem \ref{he1nbx351}, and so, we omit the details.
\end{proof}
\section{Explicit evaluations}
After obtaining class invariants $G_n$, and $g_{n}$, then, Theorem \ref{he1nbx351}, Corollary \ref{nb1}, and Theorem \ref{he1nb351} can be utilized to calculate several explicit  values of Ramanujan's singular moduli and Ramanujan- Selberg continued fraction. We conclude the present work with following the computations.
\begin{theorem}\label{the5.1} We have
\begin{eqnarray*}
\alpha_{36}&=& \left(\sqrt{2}-1\right)^4\left(\sqrt{3}-\sqrt{2}\right)^4 \left(\sqrt{\frac{\sqrt{3}+3}{2}}-\sqrt{\frac{\sqrt{3}+1}{2}}\right)^8, \label{expie1}\\
\alpha_{4/9}&=& \left(\sqrt{2}-1\right)^4\left(\sqrt{3}-\sqrt{2}\right)^4 \left(\sqrt{\frac{\sqrt{3}+3}{2}}+\sqrt{\frac{\sqrt{3}+1}{2}}\right)^8\label{expie1}.
\end{eqnarray*}
\end{theorem}
\begin{proof}
Letting $n=4$, $g_4=2^{1/8}$  \cite[ Theorem 4.1.2 (i)]{Yi-Thesis-2000} and  employing this value in \eqref{thmp1eqn2 }, and \eqref{them1eqn22 }, we evaluate that
\begin{eqnarray}\label{expli1}
\alpha_{36}&=& \left(3-2\sqrt{2}\right)^2\left(\sqrt{3}-\sqrt{2}\right)^4 \left(\sqrt{\frac{\sqrt{3}+3}{2}}-\sqrt{\frac{\sqrt{3}+1}{2}}\right)^8,\\
\alpha_{4/9}&=& \left(3-2\sqrt{2}\right)^2\left(\sqrt{3}-\sqrt{2}\right)^4 \left(\sqrt{\frac{\sqrt{3}+3}{2}}+\sqrt{\frac{\sqrt{3}+1}{2}}\right)^8. \label{expie1}
\end{eqnarray}
By \cite[p.284, Eq.(9.5)]{Berndt-notebook-5}, we have
\begin{eqnarray}\label{expli2}
3-2\sqrt{2}&=& \left(\sqrt{2}-1\right)^2
\end{eqnarray}
Employing \eqref{expli2} in \eqref{expli1}, and \eqref{expie1}, we arrive at desired results.
\end{proof}
\begin{theorem}\label{the5.2} We have
\begin{eqnarray}\label{72}
\alpha_{72}&=& \left(\sqrt{\frac{\sqrt{2}+2}{2}}-\sqrt{\frac{\sqrt{2}}{2}}\right)^{16}
\left(\sqrt{\frac{3\sqrt{2}+6}{2}}-\sqrt{\frac{3\sqrt{2}+4}{2}}\right)^{8},\\
\alpha_{8/9}&=& \left(\sqrt{\frac{\sqrt{2}+2}{2}}-\sqrt{\frac{\sqrt{2}}{2}}\right)^{16}
\left(\sqrt{\frac{3\sqrt{2}+6}{2}}+\sqrt{\frac{3\sqrt{2}+4}{2}}\right)^{8}. \label{89}
\end{eqnarray}
\end{theorem}
\begin{proof}
Letting $n=8$, $g_8=2^{1/8}\left(\sqrt{2}+1\right)^{1/8}$  \cite[ Theorem 4.1.2 (ii)]{Yi-Thesis-2000}. It follows that
\begin{align}\label{abikd}
\sqrt{g_8^{24}+1}= 5+4\sqrt{2}
\quad ; \quad
\sqrt{g_8^{8}+1}= \sqrt{2}+1
\quad ; \quad
\sqrt{g_8^{16}-g_8^{8}+1} = 3+\sqrt{2}.
\end{align}
Applying \eqref{abikd} in \eqref{thmp1eqn2 }, we deduce that
\begin{eqnarray}\label{al72}
\alpha_{72}&=& \left(5+4\sqrt{2}-\sqrt{56+40\sqrt{2}}\right)^2\left(\sqrt{2}+1-\sqrt{2+2\sqrt{2}}\right)^4
\left(\sqrt{\frac{3\sqrt{2}+6}{2}}-\sqrt{\frac{3\sqrt{2}+4}{2}}\right)^{8}.
\end{eqnarray}
Now we apply Lemma 9.10 \cite[p. 292]{Berndt-notebook-5} with $r = 5+4\sqrt{2}$. Then $t = (\sqrt{2}+1)/2$ and so
\begin{eqnarray}\label{al172}
5+4\sqrt{2}-\sqrt{56+40\sqrt{2}}&=& \left(\sqrt{\frac{\sqrt{2}+2}{2}}-\sqrt{\frac{\sqrt{2}}{2}}\right)^{4}.
\end{eqnarray}
Further,
\begin{eqnarray}\label{al173}
\sqrt{2}+1-\sqrt{2+2\sqrt{2}}&=& \left(\sqrt{\frac{\sqrt{2}+2}{2}}-\sqrt{\frac{\sqrt{2}}{2}}\right)^{2}.
\end{eqnarray}
From \eqref{al72}, \eqref{al172}, and \eqref{al173}, we deduce \eqref{72}. Similarly we arrive at \eqref{89}.
\end{proof}
\begin{theorem} We have
\begin{eqnarray*}
S_1\left(e^{-6\pi}\right)&=& \frac{1}{\sqrt{2}}\left(\sqrt{2}-1\right)^{1/2}\left(\sqrt{3}-\sqrt{2}\right)^{1/2} \left(\sqrt{\frac{\sqrt{3}+3}{2}}-\sqrt{\frac{\sqrt{3}+1}{2}}\right),\\
S_1\left(e^{-2\pi/3}\right)&=& \frac{1}{\sqrt{2}}\left(\sqrt{2}-1\right)^{1/2}\left(\sqrt{3}-\sqrt{2}\right)^{1/2} \left(\sqrt{\frac{\sqrt{3}+3}{2}}+\sqrt{\frac{\sqrt{3}+1}{2}}\right),\\
S_1\left(e^{-\pi6\sqrt{2}}\right)&=& \frac{1}{\sqrt{2}}\left(\sqrt{\frac{\sqrt{2}+2}{2}}-\sqrt{\frac{\sqrt{2}}{2}}\right)^{2}
\left(\sqrt{\frac{3\sqrt{2}+6}{2}}-\sqrt{\frac{3\sqrt{2}+4}{2}}\right),\\
S_1\left(e^{-\pi2\sqrt{2}/3}\right)&=& \frac{1}{\sqrt{2}}\left(\sqrt{\frac{\sqrt{2}+2}{2}}-\sqrt{\frac{\sqrt{2}}{2}}\right)^{2}
\left(\sqrt{\frac{3\sqrt{2}+6}{2}}+\sqrt{\frac{3\sqrt{2}+4}{2}}\right).
\end{eqnarray*}
\end{theorem}
\begin{proof}
The proof of theorem can be obtained by \eqref{thcv}, and \eqref{thcv1}. Since the proof is analogous to previous theorems, and so, we omit the details.
\end{proof}
\begin{theorem} We have
\begin{eqnarray*}
S_2\left(e^{-\pi3\sqrt{5}}\right)&=&\frac{1}{\sqrt{2}}
\left(\sqrt{\frac{\sqrt{5}+3}{4}}-\sqrt{\frac{\sqrt{5}-1}{4}}\right)^2
\left(\sqrt{\frac{3\sqrt{5}+7}{4}}-\sqrt{\frac{3\sqrt{5}+3}{4}}\right) ,\\
S_2\left(e^{-\pi\sqrt{5}/3}\right)&=&\frac{1}{\sqrt{2}}
\left(\sqrt{\frac{\sqrt{5}+3}{4}}-\sqrt{\frac{\sqrt{5}-1}{4}}\right)^2
\left(\sqrt{\frac{3\sqrt{5}+7}{4}}+\sqrt{\frac{3\sqrt{5}+3}{4}}\right) \\
S_2\left(e^{-\pi3\sqrt{7}}\right)&=&\frac{1}{\sqrt{2}}\left(\frac{3-\sqrt{7}}{\sqrt{2}}\right)^{1/2}\left(\frac{\sqrt{3}-1}{\sqrt{2}}\right)
\left(\sqrt{\frac{\sqrt{21}+5}{2}}-\sqrt{\frac{\sqrt{21}+3}{2}}\right) ,\\
S_2\left(e^{-\pi\sqrt{7}/3}\right)&=&\frac{1}{\sqrt{2}}\left(\frac{3-\sqrt{7}}{\sqrt{2}}\right)^{1/2}\left(\frac{\sqrt{3}-1}{\sqrt{2}}\right)
\left(\sqrt{\frac{\sqrt{21}+5}{2}}+\sqrt{\frac{\sqrt{21}+3}{2}}\right).
\end{eqnarray*}
\end{theorem}
\begin{proof}
Employing the class invariant $G_5,$ and $G_7$ (see \cite[p. 189]{Berndt-notebook-5}) in \eqref{a1eqn2}, and \eqref{a1eqn21}, we obtain all the above values. Since the proof is analogous to Theorem \ref{the5.1}, and Theorem \ref{the5.2} and so, omit the details.
\end{proof}

\end{document}